\documentclass[11pt]{amsart}

\usepackage{amsmath}
\usepackage{amssymb}
\usepackage{amsthm}
\usepackage[hidelinks]{hyperref}
\usepackage{geometry}

\calclayout

\numberwithin{equation}{section}

\newtheorem{theorem}{Theorem}[section]
\newtheorem{proposition}[theorem]{Proposition}
\newtheorem{lemma}[theorem]{Lemma}
\newtheorem{corollary}[theorem]{Corollary}

\theoremstyle{definition}
\newtheorem{definition}[theorem]{Definition}

\theoremstyle{remark}

\newtheorem{notation}[theorem]{Notation}

\newcommand{\tref}[2][]
	{\ifx#1""{Theorem~\textup{\ref{#2}}}%
	\else{Theorems~\textup{\ref{#1}} and~\textup{\ref{#2}}}\fi}
\newcommand{\trefs}[4][]
	{\ifx#1""{Theorems~\textup{\ref{#2}},~\textup{\ref{#3}} and~\textup{\ref{#4}}}%
	\else{Theorems~\textup{\ref{#1}},~\textup{\ref{#2}},~\textup{\ref{#3}} and~\textup{\ref{#4}}}\fi}
\newcommand{\pref}[2][]
	{\ifx#1""{Proposition~\textup{\ref{#2}}}%
	\else{Propositions~\textup{\ref{#1}} and~\textup{\ref{#2}}}\fi}
\newcommand{\prefs}[4][]
	{\ifx#1""{Propositions~\textup{\ref{#2}},~\textup{\ref{#3}} and~\textup{\ref{#4}}}%
	\else{Propositions~\textup{\ref{#1}},~\textup{\ref{#2}},~\textup{\ref{#3}} and~\textup{\ref{#4}}}\fi}
\newcommand{\lref}[2][]
	{\ifx#1""{Lemma~\textup{\ref{#2}}}%
	\else{Lemmas~\textup{\ref{#1}} and~\textup{\ref{#2}}}\fi}
\newcommand{\lrefs}[4][]
	{\ifx#1""{Lemma~\textup{\ref{#2}}, \textup{\ref{#3}} and~\textup{\ref{#4}}}%
	\else{Lemmas~\textup{\ref{#1}}, \textup{\ref{#2}}, \textup{\ref{#3}} and \textup{\ref{#4}}}\fi}
\newcommand{\corref}[2][]
	{\ifx#1""{Corollary~\textup{\ref{#2}}}%
	\else{Corollaries~\textup{\ref{#1}} and~\textup{\ref{#2}}}\fi}
\newcommand{\rref}[2][]
	{\ifx#1""{Remark~\textup{\ref{#2}}}%
	\else{Remark~\textup{\ref{#1}} and~\textup{\ref{#2}}}\fi}
\newcommand{\dref}[2][]
	{\ifx#1""{Definition~\textup{\ref{#2}}}%
	\else{Definition~\textup{\ref{#1}} and~\textup{\ref{#2}}}\fi}
\newcommand{\iref}[2][]
	{\ifx#1""{\textup{(\ref{#2})}}%
	\else{\textup{(\ref{#1})}~and~\textup{(\ref{#2})}}\fi}
\newcommand{\aref}[2][]
	{\ifx#1""{Assumption~\textup{\ref{#2}}}%
	\else{Assumption~\textup{\ref{#1}} and~\textup{\ref{#2}}}\fi}
\newcommand{\noteref}[2][]
	{\ifx#1""{Notation~\textup{\ref{#2}}}%
	\else{Notation~\textup{\ref{#1}} and~\textup{\ref{#2}}}\fi}
\newcommand{\secref}[2][]
	{\ifx#1""{Section~\textup{\ref{#2}}}%
	\else{Section~\textup{\ref{#1}} and~\textup{\ref{#2}}}\fi}

\newcommand{\CmplNum}{\mathbf{C}}
\newcommand{\RealNum}{\mathbf{R}}
\newcommand{\NaturalNum}{\mathbf{N}}
\newcommand{\Integers}{\mathbf{Z}}

\newcommand{\ImUnit}{\sqrt{-1}}

\newcommand{\expect}{\boldsymbol{E}}

\newcommand{\Hurst}{H}

\newcommand{\SchwartzClass}{\mathcal{S}}

\newcommand{\RP}[1]{\boldsymbol{#1}}
\newcommand{\var}[1]{{#1}\mathchar`-{\mathrm{var}}}

\newcommand{\floor}[1]{\lfloor{#1}\rfloor}

\newcommand{\indicator}{\mathtt{1}}

\newcommand{\lf}{\mathrm{lf}}
\newcommand{\hf}{\mathrm{hf}}

\begin{document}

\title[Rough path lift of mBms]{Rough path lifts of multifractional Brownian motions}

\author[N. Naganuma]{Nobuaki Naganuma}
\address[N. N.]{Faculty of Advanced Science and Technology, Kumamoto University.
2-39-1 Kurokami, Chuo-ku, Kumamoto, 860-8555, JAPAN}
\email{naganuma@kumamoto-u.ac.jp}
\thanks{\emph{Acknowledgment}. This work was supported by JSPS KAKENHI Grant Number 22K13932.}

\begin{abstract}
	We show that multifractional Brownian motions (mBms) can be lifted to geometric rough paths.
	One of the main ingredients in the proof is a decomposition of mBms
	via the Lemari{\'e}-Meyer orthogonal wavelet basis, which is introduced in the context of multiresolution analysis.
	As a consequence, we are able to construct
	solutions to rough differential equations driven by mBms.
\end{abstract}

\date{October 23, 2026}

\maketitle

\tableofcontents

\section{Introduction}
We begin by recalling fractional Brownian motion (fBm).
A one-dimensional fBm $B=\{B(t)\}_{t\in\RealNum}$
is a real-valued centered Gaussian process
with continuous sample paths starting from the origin
and with covariance satisfying
\begin{align*}
	\expect[B(s)B(t)]
	=
		\frac{c_\Hurst^2}{2}
		\{
			|s|^{2\Hurst}
			+
			|t|^{2\Hurst}
			-
			|s-t|^{2\Hurst}
		\}
	\qquad
	\text{for}
	\quad s,t\in\RealNum.
\end{align*}
Here $0<\Hurst<1$ is a constant parameter, called the Hurst parameter,
and $c_\Hurst$ is a positive constant satisfying
\begin{align}\label{eq4892084098324}
	c_\Hurst^2
	=
		\frac{1}{\Gamma(2\Hurst+1)\sin\pi \Hurst}.
\end{align}
While many readers may be familiar with $c_\Hurst=1$,
we use this normalization for consistency with 
the definition of multifractional Brownian motion (mBm),
which is the main subject of the present work.
A $d$-dimensional fBm is defined as a $d$-dimensional stochastic process
whose components are independent one-dimensional fBms.
In this paper, the Hurst parameters of the components may differ;
however, in this introduction we assume them to have the same value $\Hurst$ for simplicity.
From this definition, we see 
that an fBm with $\Hurst=\frac{1}{2}$ is a standard Brownian motion.
It is well-known that fBm has (global) self-similarity and stationary increments.
Furthermore, sample paths of fBms are $(\Hurst-\delta)$-H{\"o}lder continuous,
where $0<\delta<\Hurst$. In particular, if $\Hurst\neq\frac{1}{2}$ then an fBm is neither a semimartingale nor a Markov process.
Due to these remarkable properties,
fBm has been extensively studied 
(see \cite{BiaginiHuOksendalZhang2008Book, Mishura2008Book, Nourdin2012Book,  PipirasTaqqu2017Book}).

Next we recall Young integration theory (\cite{YoungLC1936})
and rough path analysis (\cite{Lyons1998}).
The reader is referred to \cite{FrizHairer2020Book, FrizVictoir2010Book, LyonsQian2002Book}.
Both theories are concerned with controlled ordinary differential equations (ODEs),
which are deterministic rather than probabilistic.
The choice between these two theories depends on
the roughness of the driving functions of controlled ODEs.
In Young integration theory, we can handle driving functions
with finite $p$-variation as drivers for $1<p<2$.
In rough path analysis, we can also handle driving functions
with finite $p$-variation as drivers for $p\geq 2$,
but in this case, we need to construct a (geometric) rough path associated with the driver.
Here, a rough path is a collection consisting of increments of the driver
and objects analogous to iterated integrals of the driver,
and a geometric rough path is a rough path obtained as the limit of a sequence of iterated integrals.
However, since the driver is not smooth,
we cannot define the iterated integrals in the usual way.
In what follows, we sometimes refer to constructing a rough path
associated with a function as lifting the function.

As an application of Young integration theory and rough path analysis, 
we can consider stochastic differential equations (SDEs) driven by fBms.
When $\Hurst=\frac{1}{2}$, an fBm is a standard Brownian motion
and we usually use the It{\^o} theory to formulate the notion of a solution to the SDE.
When $\Hurst\neq\frac{1}{2}$, fBms are not semimartingales so the It{\^o} theory is unavailable.
Thus, Young integration theory is typically applied in the case $\frac{1}{2}<\Hurst<1$,
while rough path analysis is used in the case $0<\Hurst\leq\frac{1}{2}$.
In the latter case, as stated above,
we construct geometric rough paths associated with fBm in a probabilistic way;
in other words, we lift the fBms to geometric rough paths.

Here we focus on how to construct rough paths
associated with Gaussian processes including fBm.
Some sufficient conditions for constructing rough paths
associated with Gaussian processes are known.
If Gaussian processes satisfy one of the sufficient conditions,
we can construct geometric rough paths associated with the processes.
The first condition is known as the Coutin-Qian condition (\cite{CoutinQian2002}).
The authors constructed geometric rough paths associated
with fBms by establishing the sufficient condition.
We precisely state this condition in \dref{def489238430212423} and use it to show our main results.
The second condition is formulated in terms of two-dimensional $\rho$-variation
of the covariance functions (\cite{FrizVictoir2010b}). 
Here $1\leq \rho<2$ and we typically choose $\rho=\frac{1}{2\Hurst}$
when we consider fBms.
This is a generalization of the Coutin-Qian condition.
A standard approach to verifying these conditions relies on explicit expressions for the covariance function.
The third condition is stated in terms of
covariance measure structures (\cite{FrizGessGulisashviliRiedel2016}).
This condition allows us to handle Gaussian processes
without requiring an explicit expression for the covariance function.
In particular, the authors of \cite{FrizGessGulisashviliRiedel2016} studied random Fourier series.
Note that a special case of the third condition is found
in \cite[Section~10.3]{FrizHairer2020Book}
as the ``concavity criterion.''
We should note that fBms can be lifted to geometric rough paths
only in the case $\frac{1}{4}<\Hurst\leq\frac{1}{2}$
by using the sufficient conditions stated above.
For $0<\Hurst<\frac{1}{4}$, although one can construct
a (nongeometric) rough path associated with fBms,
this is not the focus of the present work.

We are now in a position to introduce mBms.
In the present paper, we frequently refer to \cite{Ayache2019Book} for their details.
Historically, Peltier and L{\'e}vy V{\'e}hel (\cite{PeltierLevyVehel1995})
introduced moving-average type mBms as extensions of fBms, which are indexed by $\RealNum$,
while  Benassi, Jaffard, and Roux (\cite{BenassiJaffardRoux1997})
introduced real-harmonisable type multifractional Brownian fields indexed by $\RealNum^d$ ($d\ge 1$).
Hereafter we consider only real-harmonisable type mBms, which are indexed by $\RealNum$.
Here we roughly state the definition of mBms.
Let $h\colon\RealNum\to (0,1)$ be a $C^1$-function.
We define mBms $B=\{B(t)\}_{t\in\RealNum}$ by replacing
the Hurst parameter $\Hurst$ in an expression of fBm by the function $h$.
Of course, if $h$ is a constant function, then mBms are fBms.
Furthermore, the covariance function is given by
\begin{align*}
	\expect[B(s)B(t)]
	=
		\frac{c_{\frac{h(s)+h(t)}{2}}^2}{2}
		\{
			|s|^{h(s)+h(t)}
			+
			|t|^{h(s)+h(t)}
			-
			|s-t|^{h(s)+h(t)}
		\}
	\qquad
	\text{for}
	\quad s,t\in\RealNum.
\end{align*}
Here $c_{\frac{h(s)+h(t)}{2}}$ is a function of $s$ and $t$ defined by
replacing $\Hurst$ with $\frac{h(s)+h(t)}{2}$ in $c_\Hurst$,
where $c_\Hurst$ is a positive constant satisfying \eqref{eq4892084098324}.
From this expression, we see that, in general, 
mBms admit neither (global) self-similarity nor stationary increments.
The absence of these properties makes mBms more difficult to study.

Next, we explain our main results, which involve constructing a rough path lift of mBms.
In the preceding results, rough path lifts are based
on expressions of the covariance functions or covariance measure structures.
It seems difficult to show that the Coutin-Qian condition is applicable based
on the expression above, and it is also unclear 
whether the covariance measure structures can be analyzed.
Thus, in the present paper, we construct rough paths
associated with mBms in another way, as explained below.
We divide mBms into irregular and regular stochastic processes.
In this decomposition, we will use
the Lemari{\'e}-Meyer orthogonal wavelet basis of $L^2(\RealNum)$.
This basis is typically developed in the context of multiresolution analysis.
After the decomposition, we will show that the irregular stochastic process
satisfies the Coutin-Qian condition.
Finally, we complete the construction of a rough path
via translation in the direction of the regular stochastic process.
One advantage of this proof is that we can construct rough paths
without relying on the explicit form of the covariance functions.
Another advantage is that the tools used are standard.

We summarize the previous results and discuss directions for future work.
In \cite{Harang2020}, Harang studied rough differential equations driven by mBms. 
He introduced a variant of the space of rough paths and solved the equations.
In \cite{BianchiPantanellaPianese2012, BianchiPantanellaPianese2013, BianchiPianese2014},
the authors apply mBms to problems in mathematical finance.
A possible direction for future work is to study SDEs motivated by mathematical finance.

The organization of the present paper is as follows.
In \secref{SecFrameworkResults}, we introduce
the framework we will work in and state our main results.
\secref{SecAnalysisOfGenerator} is devoted to
the analysis of the generator of mBms 
and the proofs of our main results.
In Appendix~\ref{SecAppendix}, we collect technical lemmas.

At the end of this introduction, we present the frequently used notation.
\begin{notation}\label{notation100}
\begin{itemize}
	\item	$\floor{x}$ stands for the integer part of $x\geq 0$.
	\item	$\mathcal{B}_0(\RealNum)$ denotes 
			the collection of Borel subsets of $\RealNum$
			with finite Lebesgue measure.
	\item	$\SchwartzClass(\RealNum)$ denotes the Schwartz space on $\RealNum$.
	\item	For every $f\in L^1(\RealNum)$, the Fourier transform and its inverse are defined by
			\begin{align*}
				\hat{f}(\xi)
				&=
					\frac{1}{\sqrt{2\pi}}
					\int_\RealNum
						e^{-\ImUnit x\xi}
						f(x)\,
						dx,
				&
				\check{f}(\xi)
				&=
					\frac{1}{\sqrt{2\pi}}
					\int_\RealNum
						e^{\ImUnit x\xi}
						f(x)\,
						dx.
			\end{align*}
	\item	We denote by $c$ a constant whose value is not important and may change from line to line.
\end{itemize}
\end{notation}

\section{Framework and main results}\label{SecFrameworkResults}
In this section, we present the framework we will work in and state our main results.
\subsection{Multifractional Brownian motions}
By following \cite{Ayache2019Book}, we introduce an mBm,
which is denoted by $B=\{B(t)\}_{t\in\RealNum}$.
This process is constructed from
the generator of mBms $X=\{X(u,v)\}_{(u,v)\in\RealNum\times(0,1)}$
and a deterministic continuous function $h=\{h(t)\}_{t\in \RealNum}$
and is defined by $B(t)=X(t,h(t))$.
For each fixed $v\in(0,1)$, $X(\cdot,v)$ is an fBm with Hurst parameter $v$, 
and $h$ is a function, called the Hurst function.

We first give a precise definition of the generator of mBm.
Let $\mathbb{W}=\{\mathbb{W}(A)\}_{A\in\mathcal{B}_0(\RealNum)}$ be a Brownian measure,
that is, a real-valued centered Gaussian orthogonally scattered random measure 
on $\RealNum$ with Lebesgue control measure.
Let $\hat{\mathbb{W}}$ be an orthogonally scattered random measure,
which is defined as the Wiener integral
\begin{align*}
	\hat{\mathbb{W}}(A)
	=
		\int_\RealNum
			\hat{\indicator}_A(s)\,
			d\mathbb{W}(s)
	\qquad
	\text{for}
	\qquad
	A\in\mathcal{B}_0(\RealNum).
\end{align*}
Here $\indicator_A$ denotes the indicator function of $A$.
For more details, see \cite[Definition~1.26]{Ayache2019Book}.
Define $K\colon\RealNum\times(0,1)\times(\RealNum\setminus\{0\})\to\CmplNum$ by
\begin{align*}
	K(u,v,\xi)
	=
		\frac{1}{\sqrt{2\pi}}
		\frac{e^{\sqrt{-1}u\xi}-1}{|\xi|^{v+1/2}}.
\end{align*}
Define $X=\{X(u,v)\}_{(u,v)\in\RealNum\times(0,1)}$ by
\begin{align*}
	X(u,v)
	=
		\int_\RealNum
			K(u,v,\xi)\,
			d\hat{\mathbb{W}}(\xi).
\end{align*}
The generator of mBm is defined as
a continuous modification of $X$ and denoted by the same notation $X$
(see \cite[Proposition~1.80 and Remark 1.81]{Ayache2019Book}).
Note
\begin{align*}
	\expect[|X(1,v)|^2]
	=
		\int_\RealNum
			|K(1,v,\xi)|^2\,
			d\xi
	=
		\frac{1}{2\pi}
		\int_\RealNum
			\frac{|e^{\sqrt{-1}\xi}-1|^2}{|\xi|^{2v+1}}\,
			d\xi
	=
		\frac{1}{\Gamma (2 v + 1)\sin \pi v}
	=
		c_v^2.
\end{align*}
We also have
\begin{multline*}
	2\expect[X(u,v)X(u',v')]\\
	\begin{aligned}
	&=
		\int_\RealNum
			K(u,v,\xi)
			\overline{K(u',v',\xi)}\,
			d\xi
		+
		\int_\RealNum
			\overline{K(u,v,\xi)}
			K(u',v',\xi)\,
			d\xi\\
	&=
		\frac{1}{2\pi}
		\int_\RealNum
			\frac{|e^{\sqrt{-1}u\xi}-1|^2}{|\xi|^{v+v'+1}}\,
			d\xi
		+
		\frac{1}{2\pi}
		\int_\RealNum
			\frac{|e^{\sqrt{-1}u'\xi}-1|^2}{|\xi|^{v+v'+1}}\,
			d\xi
		-
		\frac{1}{2\pi}
		\int_\RealNum
			\frac{|e^{\sqrt{-1}(u-u')\xi}-1|^2}{|\xi|^{v+v'+1}}\,
			d\xi\\
	&=
		c_{\frac{v+v'}{2}}^2
		\big\{
			|u|^{v+v'}
			+|u'|^{v+v'}
			-|u-u'|^{v+v'}
		\big\}.
\end{aligned}
\end{multline*}
Hence, we see
\begin{align*}
	\expect[X(u,v)X(u',v)]
	=
		\frac{c_v^2}{2}
		\left\{
			|u|^{2v}
			+|u'|^{2v}
			-|u-u'|^{2v}
		\right\},
\end{align*}
which implies $X(\bullet,v)$ is an fBm with the Hurst parameter $v$
in the sense of the definition given in the Introduction.
We define an mBm as an extension of an fBm as follows.
\begin{definition}
	Let $h\colon\RealNum\to(0,1)$ be a deterministic $C^1$-function
	and set $B(t)=X(t,h(t))$. Then we call $B=\{B(t)\}_{t\in\RealNum}$
	a multifractional Brownian motion (mBm) and $h=\{h(t)\}_{t\in\RealNum}$ a Hurst function.
	A $d$-dimensional mBm is defined as a $d$-dimensional stochastic process
	whose components are independent one-dimensional mBms.
\end{definition}

Note that the components of a $d$-dimensional mBm
may have different Hurst functions.

Next we consider series expansions of one-dimensional processes $X$ and $B$.
We refer the reader to \cite[Sections~5 and 6]{Ayache2019Book}.
We begin the discussion with wavelet analysis.
Let $\{\psi_{jk}\}_{j,k\in\Integers}$ denote the Lemari{\'e}-Meyer orthogonal wavelet basis of $L^2(\RealNum)$.
Although the basis is typically developed in the context of multiresolution analysis, 
we define it directly without going into the details of that theory.
For more details, see \cite[Section~4]{Ayache2019Book}.
Let $\nu\colon\RealNum\to[0,1]$ be an infinitely differentiable function satisfying
\begin{align*}
	\nu(\xi)
	&=
		\begin{cases}
			0,	&	\xi\leq 0,\\
			1,	&	\xi\geq 1,
		\end{cases}
	&
	\nu(\xi)+\nu(1-\xi)=1,
	\qquad
	\xi\in\RealNum.
\end{align*}
We define a function $\psi\in\SchwartzClass(\RealNum)$ so that
$\hat{\psi}$ is an infinitely differentiable compactly supported function 
vanishing in a neighborhood of $0$ satisfying
\begin{align*}
	\hat{\psi}(\xi)
	=
		\frac{1}{\sqrt{2\pi}}
		e^{\ImUnit\frac{\xi}{2}}
		\begin{cases}
			\sin
				\left[
					\frac{\pi}{2}
					\nu
						\left(
							\frac{3}{2\pi}|\xi|-1
						\right)
				\right],
			&
				\frac{2}{3}\pi\leq |\xi|\leq \frac{4}{3}\pi,\\
			\cos
				\left[
					\frac{\pi}{2}
					\nu
						\left(
							\frac{3}{4\pi}|\xi|-1
						\right)
				\right],
			&
				\frac{4}{3}\pi\leq |\xi|\leq \frac{8}{3}\pi,\\
			0,
			& 
				\text{otherwise}.
		\end{cases}
\end{align*}
By using $\psi$, we set
\begin{align*}
	\psi_{jk}(x)=2^{j/2}\psi(2^j x-k).
\end{align*}

With the basis $\{\psi_{jk}\}_{j,k\in\Integers}$ in hand, 
we introduce a series expansion of $X$.
To this end, we define independent standard Gaussian random variables
$\{\varepsilon_{jk}\}_{j,k\in\Integers}$
and a deterministic function $\Psi$ defined on $\RealNum\times\RealNum$ as follows:
\begin{align*}
	\varepsilon_{jk}
	&=
		\int_\RealNum
			\overline{\hat{\psi}_{jk}(\xi)}\,
			d\hat{\mathbb{W}}(\xi),
	&
	\Psi(y,v)
	&=
		\frac{1}{\sqrt{2\pi}}
		\int_\RealNum
			e^{\sqrt{-1}y\eta}
			\frac{\hat{\psi}(\eta)}{|\eta|^{v+1/2}}\,
			d\eta.
\end{align*}
Then we have
\begin{align*}
	X(u,v)
	=
		\sum_{j,k\in\Integers}
			2^{-jv}
			\varepsilon_{jk}
			\{
				\Psi(2^ju-k,v)
				-
				\Psi(-k,v)
			\}
	\qquad
	\text{a.s.}
\end{align*}
Note that, with probability $1$, the series on the right-hand side converges
uniformly in $(u,v)$ on each compact rectangle of $\RealNum\times(0,1)$.
See \cite[Theorem~5.5]{Ayache2019Book}.
Here we divide $X$ into high and low frequency parts as follows:
\begin{align*}
	X^\hf(u,v)
	=
		\sum_{j\geq 1}
		\sum_{k\in\Integers}
			2^{-jv}
			\varepsilon_{jk}
			\{
				\Psi(2^ju-k,v)
				-
				\Psi(-k,v)
			\},\\
	X^\lf(u,v)
	=
		\sum_{j\leq 0}
		\sum_{k\in\Integers}
			2^{-jv}
			\varepsilon_{jk}
			\{
				\Psi(2^ju-k,v)
				-
				\Psi(-k,v)
			\}.
\end{align*}
Furthermore we divide $X^\hf$ into two parts as follows:
\begin{align*}
	X^\hf_+(u,v)
	&=
		\sum_{j\geq 1}
		\sum_{k\in\Integers}
			2^{-jv}
			\varepsilon_{jk}
			\Psi(2^ju-k,v),
	&
	X^\hf_-(u,v)
	&=
		-
		\sum_{j\geq 1}
		\sum_{k\in\Integers}
			2^{-jv}
			\varepsilon_{jk}
			\Psi(-k,v).
\end{align*}
Then we see that $X^\lf(u,v)$ and $X^\hf_-(u,v)$ are 
smooth stochastic processes in $u$ and $v$,
which follows from \cite[Proposition~6.1]{Ayache2019Book}.
The roughness of $X$ is present only in $X^\hf_+$.
We set 
\begin{align*}
	B^\hf_+(t)&=X^\hf_+(t,h(t)), & 
	B^\hf_-(t)&=X^\hf_-(t,h(t)), &
	B^\lf(t)&=X^\lf(t,h(t)).
\end{align*}
Then we see $B(t)=B^\hf_+(t)+B^\hf_-(t)+B^\lf(t)$ a.s.
This decomposition plays an important role in our study.

\subsection{Rough path analysis}
Next we collect elements of rough path analysis.
This part follows from \cite{LyonsQian2002Book}.
Let $T>0$ and set $I=[0,T]$ and $\triangle=\{(s,t)\mid \text{$s,t\in I$ with $s<t$}\}$.

We denote by $V$ a Banach space.
Let $p\geq 1$.
For $z\in C(I,V)$, we introduce the $p$-variation seminorm as follows:
\begin{align*}
	\|z\|_{\var{p}}
	=
		\left\{
			\sup
				\sum_{i=1}^n
					|z_{t_{i-1},t_i}|^p
		\right\}^{\frac{1}{p}},
\end{align*}
where the supremum runs over all finite partitions $0=t_0<\dots<t_n=T$ of the interval $I$
and $z_{s,t}=z_t-z_s$ for $s,t\in I$ with $s<t$.
For $z\in C(\triangle,V)$, we define $\|z\|_{\var{p}}$ in the same way.
For $\Theta=I$ or $\triangle$, we denote by $C^{\var{p}}(\Theta,V)$
the space of all functions $z\in C(\Theta,V)$ satisfying $\|z\|_{\var{p}}<\infty$.

We first define the largest space of rough paths as follows:
\begin{definition}
	Let $p\geq 2$.
	Let $\RP{z}=(1,\RP{z}^1,\dots,\RP{z}^{\floor{p}})\in C(\triangle,\bigoplus_{k=0}^{\floor{p}}(\RealNum^d)^{\otimes k})$.
	If $\RP{z}$ satisfies
	\begin{enumerate}
		\item	$\|\RP{z}^k\|_{\var{\frac{p}{k}}}<\infty$ for all $k=1,\dots,\floor{p}$,
		\item	$\RP{z}_{s,t}=\RP{z}_{s,u}\otimes \RP{z}_{u,t}$ for all $s,t,u\in I$ with $s<u<t$,
	\end{enumerate}
	then we say that $\RP{z}$ is a rough path with roughness $p$.

	We denote by $\Omega_{\var{p}}(\RealNum^d)$ the space of all rough paths
	and define the metric as follows:
	\begin{align*}
		d_{\var{p}}(\RP{z},\tilde{\RP{z}})
		=
			\sum_{k=1}^{\floor{p}}
				\|\RP{z}^k-\tilde{\RP{z}}^k\|_{\var{\frac{p}{k}}}.
	\end{align*}
	Then the space $(\Omega_{\var{p}}(\RealNum^d),d_{\var{p}})$
	is a complete metric space.
\end{definition}

We say $\RP{z}\in\Omega_{\var{p}}(\RealNum^d)$ is a (rough path) lift of $z\in C(I,\RealNum^d)$
if $\RP{z}^1_{s,t}=z_t-z_s$ for all $s,t\in I$ with $s<t$.
A typical example of a rough path is the canonical lift of a smooth function,
which is defined as follows:
\begin{definition}
	Let $p\geq 2$.
	Let $z\in C^{\var{1}}(I,\RealNum^d)$.
	Define $\RP{z}=(1,\RP{z}^1,\dots,\RP{z}^{\floor{p}})$ by $\RP{z}^1_{s,t}=z_t-z_s$ and
	$
		\RP{z}^k_{s,t}
		=
			\int_s^t
				\RP{z}^{k-1}_{s,u}\,
				\otimes
				dz_u
	$
	for $2\leq k\leq \floor{p}$.
	We call $\RP{z}$ a canonical lift of $z$.
	The space of all canonical lifts is denoted by $S\Omega_{\var{p}}(\RealNum^d)$.
\end{definition}

A geometric rough path defined as follows plays an important role in rough path analysis.
\begin{definition}
	Let $p\geq 2$.
	We set 
	\begin{align*}
		G\Omega_{\var{p}}(\RealNum^d)=\overline{S\Omega_{\var{p}}(\RealNum^d)}^{d_{\var{p}}}.
	\end{align*}
	An element of $G\Omega_{\var{p}}(\RealNum^d)$ is called a geometric rough path.
\end{definition}

We use the following elementary operation on rough paths to show our main result.
We can find the statement below in \cite{InahamaNaganuma2021}.
\begin{proposition}\label{prop473287432987491}
	Let $2\leq p<4$ and $1\leq q<2$ and assume $\frac{1}{p}+\frac{1}{q}>1$.
	Let $\RP{x}\in G\Omega_{\var{p}}(\RealNum^d)$ and $y\in C^{\var{q}}(I,\RealNum^d)$.
	Define $\RP{z}=(1,\RP{z}^1,\dots,\RP{z}^{\floor{p}})$ by
	\begin{gather*}
		\RP{z}^1_{s,t}
		=
			\RP{x}^1_{s,t}+\RP{y}^1_{s,t},
		\qquad\qquad
		\RP{z}^2_{s,t}
		=
			\RP{x}^2_{s,t}
			+
			\int_s^t
				\RP{x}^1_{s,u}
				\otimes
				dy_u
			+
			\int_s^t
				\RP{y}^1_{s,u}
				\otimes
				dx_u
			+\RP{y}^2_{s,t},\\
		\begin{aligned}
			\RP{z}^3_{s,t}
			&=
				\RP{x}^3_{s,t}
				+
				\int_s^t
					\left\{
						\int_s^v
							\RP{x}^1_{s,u}
							\otimes
							dy_u
					\right\}
					\otimes
					dx_v
				-
				\int_s^t
					\RP{y}^1_{s,v}
					\otimes
					d\RP{x}^2_{v,t}
				+
				\int_s^t
					\RP{y}^2_{s,v}
					\otimes
					dx_v\\
			&\phantom{=}
				\qquad
				+
				\int_s^t
					\RP{x}^2_{s,v}
					\otimes
					dy_v
				-
				\int_s^t
					\RP{x}^1_{s,v}
					\otimes
					d\RP{y}^2_{v,t}
				+
				\int_s^t
					\left\{
						\int_s^v
							\RP{y}^1_{s,u}
							\otimes
							dx_u
					\right\}
					\otimes
					dy_v
				+
				\RP{y}^3_{s,t}.
		\end{aligned}
	\end{gather*}
	Here $x_u=\RP{x}^1_{0,u}$ and
	$\RP{y}\in G\Omega_{\var{p}}(\RealNum^d)$ denotes the canonical lift of $y$
	and the integrals are understood in the Young sense.
	Note that $\RP{z}^3$ is defined for $3\leq p<4$.
	Then $\RP{z}\in G\Omega_{\var{p}}(\RealNum^d)$.
\end{proposition}

Next we introduce a sufficient condition for lifting Gaussian processes to rough paths.

\begin{definition}\label{def489238430212423}
	Let $\frac{1}{4}<\lambda\leq \frac{1}{2}$.
	Let $Z=\{Z(t)\}_{t\in I}$ be a real-valued centered Gaussian process
	with a continuous sample path.
	We say that $Z$ satisfies the Coutin-Qian condition with parameter $\lambda$
	if the following conditions are satisfied:
	\begin{enumerate}
		\item\label{item2121001}
				$
					\expect[|Z_t-Z_s|^2]
					\leq
						c(t-s)^{2\lambda}
				$
				for all $s,t\in I$ with $s<t$.
		\item\label{item2121002}
				$
					|\expect[\{Z(s+\epsilon)-Z(s)\}\{Z(t+\epsilon)-Z(t)\}]|
					\leq
						c
						(t-s)^{2\lambda-2}
						\epsilon^2
				$
				for all $s,t\in I$ and $\epsilon>0$
				with $\epsilon<t-s$ and $s+\epsilon,t+\epsilon\in I$.
	\end{enumerate}
	Here $c$ is a positive constant independent of $s$, $t$, and $\epsilon$.
\end{definition}

\begin{theorem}
	Let $Z=\{Z(t)=(Z^1(t),\dots,Z^d(t))\}_{t\in I}$ be a $d$-dimensional Gaussian process
	whose components are independent.
	Assume that all components satisfy the Coutin-Qian condition
	with a common parameter $\frac{1}{4}<\lambda\leq \frac{1}{2}$.
	Then we can construct a geometric rough path $\RP{Z}\in G\Omega_{\var{p}}(\RealNum^d)$
	associated with $Z$.
	Here $p$ is chosen so that $1/\lambda<p<\floor{1/\lambda}+1$.
\end{theorem}

\subsection{Main results}
Here we restrict the time parameter $t$ of mBm
to lie in the compact interval $I=[0,T]$, where $T>0$,
and introduce the notation to denote the infimum over $I$ of Hurst functions.
Roughly speaking, $\lambda_\ast=\lambda_\ast(I)$
denotes the infimum of the component Hurst functions of a $d$-dimensional mBm over $I$.

Let $B=\{B(t)\}_{t\in I}$ be a one-dimensional mBm with Hurst function $h=\{h(t)\}_{t\in I}$.
Set
\begin{align*}
	\lambda_\ast
	=
		\lambda_\ast(I)
	=
		\inf_{t\in I} h(t).
\end{align*}
For notational simplicity, we write $\lambda_\ast$ instead of $\lambda_\ast(I)$.

Let $B=\{B(t)=(B^1(t),\dots,B^d(t))\}_{t\in I}$ be a $d$-dimensional mBm
and denote by $h^\alpha=\{h^\alpha(t)\}_{t\in I}$
the Hurst function for the $\alpha$th component $B^\alpha=\{B^\alpha(t)\}_{t\in I}$.
Set
\begin{align*}
	\lambda_\ast^\alpha
	&=
		\lambda^\alpha_\ast(I)
	=
		\inf_{t\in I} h^\alpha(t),
	&
	\lambda_\ast
	&=
		\lambda_\ast(I)
	=
		\min\{\lambda_\ast^1,\dots,\lambda_\ast^d\}.
\end{align*}
Then the following is the first main result.
\begin{lemma}\label{lemma8453908409128}
	Let $B=\{B(t)\}_{t\in I}$ be a one-dimensional mBm.
	The constant $\lambda_\ast$ for $B$ is defined as above.
	Assume that $\lambda_\ast>\frac{1}{4}$.
	The process $B^\hf_+=\{B^\hf_+(t)\}_{t\in I}$ satisfies the Coutin-Qian condition
	for $\frac{1}{4}<\lambda<\min \{\frac{1}{2},\lambda_\ast\}$.
\end{lemma}

Then we see that a $d$-dimensional mBm can be lifted to a geometric rough path as follows:
\begin{theorem}\label{thm940239401923041}
	Let $B=\{B(t)\}_{t\in I}$ be a $d$-dimensional mBm.
	The constant $\lambda_\ast$ for $B$ is defined as above.
	Assume that $\lambda_\ast>\frac{1}{4}$.
	Choose $\frac{1}{4}<\lambda<\min\{\frac{1}{2},\lambda_\ast\}$.
	Then we can construct a geometric rough path $\RP{B}\in G\Omega_{\var{p}}(\RealNum^d)$
	associated with $B$.
	Here $p$ is chosen so that $1/\lambda<p<\floor{1/\lambda}+1$.
\end{theorem}
We prove \lref{lemma8453908409128} in the next section.
Assuming this lemma, we prove \tref{thm940239401923041} as follows:
\begin{proof}[Proof of \tref{thm940239401923041}]
	For every $1\leq\alpha\leq d$, we define $(B^\alpha)^\hf_+=\{(B^\alpha)^\hf_+(t)\}_{t\in I}$
	in the same way as in the one-dimensional case.
	Write $B^\hf_+=\{B^\hf_+(t)=((B^1)^\hf_+(t),\dots,(B^d)^\hf_+(t))\}_{t\in I}$.
	From \lref{lemma8453908409128}, we see that all components of $B^\hf_+$ satisfy
	the Coutin-Qian condition for $\lambda$.
	Hence there exists $\RP{B}^\hf_+\in G\Omega_{\var{p}}(\RealNum^d)$.
	Here recall $Y(t)=B(t)-B^\hf_+(t)$ is continuously differentiable and hence of finite variation.
	Using \pref{prop473287432987491}, we see the conclusion.
\end{proof}

\section{The generator of mBm and lemmas}\label{SecAnalysisOfGenerator}
By following \cite{Ayache2019Book}, we introduce
\begin{align*}
	B_j(y,v)
	&=
		\sum_{k\in\Integers}
			\varepsilon_{jk}
			\Psi(y-k,v),
	&
	\tilde{A}_j(u,v)
	&=
		2^{-jv}
		B_j(2^ju,v).
\end{align*}
Then we have
\begin{align}\label{eq493209201}
	X^\hf_+(u,v)
	=
		\sum_{j\geq 1}
			\tilde{A}_j(u,v).
\end{align}
Recall that $\Psi$ and $B_j$ are infinitely differentiable in $y$ and $v$.
See also \pref{prop843092123} below and \cite[Lemmas~6.11 and 6.12]{Ayache2019Book}.
Hence $\tilde{A}_j$ is infinitely differentiable in $u$ and $v$.

\subsection{Estimates of $\Psi$}
Here we collect propositions concerning $\Psi$.
In the following assertions, the parameter $v$ runs over a compact interval $K\subset\RealNum$.
\begin{proposition}[{\cite[Propositions~5.10 and 5.11]{Ayache2019Book}}]\label{prop843092123}
	We have $\Psi\in C^{\infty,\infty}(\RealNum\times\RealNum)$.
	Let $m,n\in\NaturalNum\cup\{0\}$.
	For all $(y,v)\in\RealNum\times K$, we have
	\begin{align*}
		|\partial_y^m\partial_v^n \Psi(y,v)|
		\leq
			c(3+|y|)^{-L}.
	\end{align*}
	Here, $L$ is an arbitrary positive constant,
	while $c$ is a positive constant that may depend on $L$ but not on $y$ or $v$.
\end{proposition}

\begin{corollary}\label{cor5984105235}
	Let $m,n\in\NaturalNum\cup\{0\}$.
	For all $(y',v')$, $(y,v)\in\RealNum\times K$, and $k\in\Integers$,
	we have
	\begin{align*}
		|\partial_y^m\partial_v^n\Psi(y'-k,v')-\partial_y^m\partial_v^n\Psi(y-k,v)|
		\leq
			c
			\{
				|y'-y|
				+
				|v'-v|
			\}.
	\end{align*}
	Here $c$ is a positive constant independent of $y', y, v', v$, and $k$.
\end{corollary}
\begin{proof}
	The fundamental theorem of calculus yields
	\begin{multline*}
		\partial_y^m\partial_v^n\Psi(y'-k,v')-\partial_y^m\partial_v^n\Psi(y-k,v)\\
		\begin{aligned}
			&=
				(y'-y)
				\int_0^1
					\partial_y^{m+1}\partial_v^n \Psi
						(
							\theta (y'-k)+(1-\theta)(y-k),
							\theta v'+(1-\theta)v
						)\,
					d\theta\\
			&\phantom{=}
				\qquad
				+
				(v'-v)
				\int_0^1
					\partial_y^m\partial_v^{n+1} \Psi
						(
							\theta (y'-k)+(1-\theta)(y-k),
							\theta v'+(1-\theta)v
						)\,
					d\theta.
		\end{aligned}
	\end{multline*}
	Combining this identity with \pref{prop843092123},
	we see the assertion.
\end{proof}

\subsection{Estimates of $B_j$}
Here we collect estimates of $B_j$ ($j\in\Integers$).
In the following assertions, the parameter $v$ runs over a compact interval $K\subset\RealNum$.
\begin{proposition}\label{prop4980239012423}
	Let $m,n\in\NaturalNum\cup\{0\}$.
	Then, we have
	\begin{align*}
		\sup_{j\in\Integers}
		\sup_{(y,v)\in\RealNum\times K}
			\expect[\{\partial_y^m\partial_v^n B_j(y,v)\}^2]
		<
			\infty.
	\end{align*}
\end{proposition}
\begin{proof}
	We have
	\begin{align*}
		\expect[\{\partial_y^m\partial_v^n B_j(y,v)\}^2]
		=
			\expect
				\left[
					\left\{
						\sum_{k\in\Integers}
							\varepsilon_{jk}
							\partial_y^m\partial_v^n
								\Psi(y-k,v)
					\right\}^2
				\right]
		=
			\sum_{k\in\Integers}
				\{\partial_y^m\partial_v^n \Psi(y-k,v)\}^2.
	\end{align*}
	\pref{prop843092123} yields
	\begin{align*}
		\expect[\{\partial_y^m\partial_v^n B_j(y,v)\}^2]
		\leq
			\sum_{k\in\Integers}
				\{c(3+|y-k|)^{-1}\}^2.
	\end{align*}
	This and \lref{lemma849208094238902} complete the proof.
\end{proof}

\begin{proposition}\label{prop483920184029841}
	Let $m,n\in\NaturalNum\cup\{0\}$.
	Then for all $(y',v'),(y,v)\in\RealNum\times K$, we have
	\begin{align*}
		\expect[\{\partial_y^m\partial_v^n B_j(y',v')-\partial_y^m\partial_v^n B_j(y,v)\}^2]
		&\leq
			c
			\{
				|y'-y|^{2\lambda}
				+
				|v'-v|^2
			\}.
	\end{align*}
	Here, $0<\lambda<1$ is an arbitrary constant,
	and $c$ is a positive constant that may depend on $\lambda$ but not on $j,y',v',y$, or $v$.
\end{proposition}
\begin{proof}
	First of all, the H{\"o}lder inequality implies
	\begin{align*}
		\expect[\{\partial_y^m\partial_v^n B_j(y',v')-\partial_y^m\partial_v^n B_j(y,v)\}^2]^{\frac{1}{2}}
		&\leq
			\expect[\{\partial_y^m\partial_v^n B_j(y',v')-\partial_y^m\partial_v^n B_j(y,v')\}^2]^{\frac{1}{2}}\\
		&\phantom{=}\qquad
			+
			\expect[\{\partial_y^m\partial_v^n B_j(y,v')-\partial_y^m\partial_v^n B_j(y,v)\}^2]^{\frac{1}{2}}.
	\end{align*}
	We estimate the two terms on the right-hand side separately, starting with the second term:
	\begin{multline*}
		\expect[\{\partial_y^m\partial_v^n B_j(y,v')-\partial_y^m\partial_v^n B_j(y,v)\}^2]
		=
			\sum_{k\in\Integers}
				\{
					\partial_y^m\partial_v^n
						\Psi(y-k,v')
					-
					\partial_y^m\partial_v^n
						\Psi(y-k,v)
				\}^2\\
		\begin{aligned}
			&=
				\sum_{k\in\Integers}
					\left\{
						(v'-v)
						\int_0^1
							\partial_y^m\partial_v^{n+1}
								\Psi(y-k,\theta v'+(1-\theta)v)\,
								d\theta
					\right\}^2\\
			&\leq
				|v'-v|^2
				\sum_{k\in\Integers}
					\{
						c(3+|y-k|)^{-1}
					\}^2.
		\end{aligned}
	\end{multline*}
	Here we used \pref{prop843092123}.
	This and \lref{lemma849208094238902}
	give an estimate of the second term.
	
	Next we give an estimate of the first term.
	We have
	\begin{align*}
		\expect[\{\partial_y^m\partial_v^n B_j(y',v')-\partial_y^m\partial_v^n B_j(y,v')\}^2]
		=
			\sum_k
				\{
					\partial_y^m\partial_v^n
						\Psi(y'-k,v')
					-
					\partial_y^m\partial_v^n
						\Psi(y-k,v')
				\}^2.
	\end{align*}
	\pref{prop843092123} implies
	\begin{align*}
		|
			\partial_y^m\partial_v^n \Psi(y'-k,v')
			-
			\partial_y^m\partial_v^n \Psi(y-k,v')
		|
		&\leq
			|\partial_y^m\partial_v^n \Psi(y'-k,v')|
			+
			|\partial_y^m\partial_v^n \Psi(y-k,v')|\\
		&\leq
			c(3+|y'-k|)^{-L}
			+
			c(3+|y-k|)^{-L}.
	\end{align*}
	By choosing $L>1$ so that $(2-2\lambda)L>1$ and using this estimate and \corref{cor5984105235},
	we have
	\begin{multline*}
		|
			\partial_y^m\partial_v^n \Psi(y'-k,v')
			-
			\partial_y^m\partial_v^n \Psi(y-k,v')
		|^2\\
		\begin{aligned}
			&=
				|
					\partial_y^m\partial_v^n \Psi(y'-k,v')
					-
					\partial_y^m\partial_v^n \Psi(y-k,v')
				|^{2-2\lambda}\\
			&\phantom{=}\qquad\quad
				\times
				|
					\partial_y^m\partial_v^n \Psi(y'-k,v')
					-
					\partial_y^m\partial_v^n \Psi(y-k,v')
				|^{2\lambda}\\
			&\leq
			\left\{
				\frac{c}{\{3+|y'-k|\}^L}
				+
				\frac{c}{\{3+|y-k|\}^L}
			\right\}^{2-2\lambda}
			\{
				c
				|y'-y|
			\}^{2\lambda}\\
		&\leq
			c
			\left\{
				\frac{1}{\{3+|y'-k|\}^{(2-2\lambda)L}}
				+
				\frac{1}{\{3+|y-k|\}^{(2-2\lambda)L}}
			\right\}
			|y'-y|^{2\lambda}.
		\end{aligned}
	\end{multline*}
	Hence, with the help of \lref{lemma849208094238902},
	we arrive at the assertion.
\end{proof}

\begin{proposition}\label{prop32132183190}
	Let $m,n,m',n'\in\NaturalNum\cup\{0\}$.
	Then for all $(y,v),(y',v')\in\RealNum\times K$ with $y\neq y'$, we have
	\begin{align*}
		|
			\expect
				[
					\partial_y^m\partial_v^n B_j(y,v)
					\partial_y^{m'}\partial_v^{n'} B_j(y',v')
				]
			|
		\leq
			c
			|y-y'|^{-L}.
	\end{align*}
	Here, $L$ is an arbitrary positive constant,
	and $c$ is a positive constant that may depend on $L$ but not on $j,y,v,y'$, or $v'$.
\end{proposition}
\begin{proof}
	We have
	\begin{align*}
		\expect
			[
				\partial_y^m\partial_v^n B_j(y,v)
				\partial_y^{m'}\partial_v^{n'} B_j(y',v')
			]
		&=
			\sum_{k\in\Integers}
				\partial_y^m\partial_v^n \Psi(y-k,v)
				\partial_y^{m'}\partial_v^{n'} \Psi(y'-k,v').
	\end{align*}
	Here we choose $\kappa>0$ so that $\kappa+L>1$.
	\pref{prop843092123} yields
	\begin{align*}
		|
			\expect
				[
					\partial_y^m\partial_v^n B_j(y,v)
					\partial_y^{m'}\partial_v^{n'} B_j(y',v')
				]
		|
		&\leq
			c
			\sum_{k\in\Integers}
				(3+|y-k|)^{-\kappa-L}
				(3+|y'-k|)^{-\kappa-L}.
	\end{align*}
	This and \lref{lem483901134} complete the proof.
\end{proof}

\subsection{Estimates of $\tilde{A}_j$}
We collect estimates of $\tilde{A}_j$ ($j\in\Integers$).
In the following assertions, 
the parameters $u$ and $v$ run over compact intervals
$I=[0,T]$ and  $K\subset(0,1)$, respectively.
Here $T$ is a fixed positive number.
\begin{proposition}\label{prop493024920921}
	Let $m,n\in\NaturalNum\cup\{0\}$.
	For all $(u,v)\in I\times K$, we have
	\begin{align*}
		\expect[\{\partial_u^m\partial_v^n \tilde{A}_j(u,v)\}^2]
		&\leq
			c
			(|j|+1)^{2n}
			2^{2j(m-v)}.
	\end{align*}
	Here $c$ is a positive constant independent of $u$, $v$, and $j$.
\end{proposition}
\begin{proof}
	Note $\partial_u^m \tilde{A}_j(u,v)=2^{j(m-v)} (\partial_y^m B_j)(2^ju,v)$.
	Here, $\partial_y$ denotes the differentiation of $B_j$ in the first variable.
	Hence, we have
	\begin{align*}
		\partial_u^m \partial_v^n \tilde{A}_j(u,v)
		&=
			\sum_{r=0}^n
				\binom{n}{r}
					(-\log 2)^r
					j^r
					2^{j(m-v)}
					\partial_y^m\partial_v^{n-r}B_j(2^ju,v).
	\end{align*}
	This implies
	\begin{align*}
		\expect[\{\partial_u^m \partial_v^n \tilde{A}_j(u,v)\}^2]^{\frac{1}{2}}
		&\leq
			c
			(|j|+1)^n
			2^{j(m-v)}
			\sup_{0\leq r\leq n}
				\expect
					[
						\{
							\partial_y^m \partial_v^r B_j(2^ju,v)
						\}^2
					]^{\frac{1}{2}}.
	\end{align*}
	Combining this estimate and \pref{prop4980239012423},
	we arrive at the conclusion.
\end{proof}

\begin{proposition}\label{prop9842309340}
	Let $n\in\NaturalNum\cup\{0\}$.
	For $u',u\in I$, and $v\in K$, we have
	\begin{align*}
		\expect[\{\partial_v^n \tilde{A}_j(u',v)-\partial_v^n \tilde{A}_j(u,v)\}^2]
		&\leq
			c
			(|j|+1)^{2n}
			2^{-2j(v-\lambda)}
			|u'-u|^{2\lambda}.
	\end{align*}
	Here, $0<\lambda<1$ is an arbitrary constant,
	and $c$ is a positive constant that may depend on $\lambda$
	but not on $u'$, $u$, $v$ or $j$.
\end{proposition}
\begin{proof}
	We have
	\begin{align*}
		\expect[\{\partial_v^n \tilde{A}_j(u',v)-\partial_v^n \tilde{A}_j(u,v)\}^2]
		&\leq
			c
			(|j|+1)^{2n}
			2^{-2jv}\\
		&\phantom{\leq}
			\qquad
			\times
			\sup_{0\leq r\leq n}
				\expect[
					\{
						\partial_v^r B_j(2^ju',v)
						-
						\partial_v^r B_j(2^ju,v)
					\}^2].
	\end{align*}
	Combining this estimate and \pref{prop483920184029841},
	we arrive at the conclusion.
\end{proof}

\begin{proposition}\label{prop438214234143}
	Write $\lambda_\ast=\min K$.
	Let $n\in\NaturalNum\cup\{0\}$.
	Assume that $j\geq 1$.
	For $u\in I$, and $v',v\in K$, we have
	\begin{align*}
		\expect[\{\partial_v^n \tilde{A}_j(u,v')-\partial_v^n \tilde{A}_j(u,v)\}^2]
		&\leq
			c
			(j+1)^{2(n+1)}
			2^{-2j \lambda_\ast}
			|v'-v|^2.
	\end{align*}
	Here $c$ is a positive constant independent of $u$, $v$, $v'$, and $j$.
\end{proposition}
\begin{proof}
	Note
	\begin{align*}
		\partial_v^n \tilde{A}_j(u,v')-\partial_v^n \tilde{A}_j(u,v)
		=
			(v'-v)
			\int_0^1
				\partial_v^{n+1} \tilde{A}_j(u,\theta v'+(1-\theta)v)\,
				d\theta.
	\end{align*}
	Using \pref{prop493024920921}, we see the conclusion.
\end{proof}

\begin{proposition}\label{prop49320431941}
	Write $\lambda_\ast=\min K$ and choose $0<\lambda<\min\{\frac{1}{2},\lambda_\ast\}$.
	Assume that $j\geq 1$.
	For all $u_1,u_2\in I$, $v_1,v_1',v_2,v_2'\in K$, 
	and $\epsilon>0$ with $\epsilon<u_2-u_1$ and $u_1+\epsilon,u_2+\epsilon\in I$,
	we have
	\begin{multline*}
		|
			\expect
				[
					\{\tilde{A}_j(u_1+\epsilon,v_1')-\tilde{A}_j(u_1,v_1)\}
					\{\tilde{A}_j(u_2+\epsilon,v_2')-\tilde{A}_j(u_2,v_2)\}
				]
		|\\
		\leq
			c
			2^{-\kappa j}
			(u_2-u_1)^{2\lambda-2}
			\{|v_1'-v_1|+\epsilon\}
			\{|v_2'-v_2|+\epsilon\}.
	\end{multline*}
	Here $c$ and $\kappa$ are positive constants
	independent of $u_1$, $u_2$, $v_1$, $v_1'$, $v_2$, $v_2'$, $\epsilon$, and $j$.
\end{proposition}
\begin{proof}
	For notational simplicity, we write 
	$u_i(\theta_i)=\theta_i (u_i+\epsilon)+(1-\theta_i)u_i=u_i+\theta_i \epsilon$
	and $v_i(\theta_i)=\theta_i v_i'+(1-\theta_i)v_i=v_i+\theta_i(v_i'-v_i)$ for $i=1,2$.
	Since the fundamental theorem of calculus implies
	\begin{align*}
		\tilde{A}_j(u_i+\epsilon,v_i')-\tilde{A}_j(u_i,v_i)
		&=
			\int_0^1
				d\theta_i\,
				\{
					\epsilon
					\partial_u \tilde{A}_j(u_i(\theta_i),v_i(\theta_i))
					+
					(v_i'-v_i)
					\partial_v \tilde{A}_j(u_i(\theta_i),v_i(\theta_i))
				\},
	\end{align*}
	we have
	\begin{multline*}
		\expect
			[
				\{\tilde{A}_j(u_1+\epsilon,v_1')-\tilde{A}_j(u_1,v_1)\}
				\{\tilde{A}_j(u_2+\epsilon,v_2')-\tilde{A}_j(u_2,v_2)\}
			]\\
			=
				\int_0^1
					d\theta_1
				\int_0^1
					d\theta_2\,
					\{
						\epsilon^2
						\Xi_{uu}(\theta_1,\theta_2)
						+
						\epsilon(v_2'-v_2)
						\Xi_{uv}(\theta_1,\theta_2)\\
						+
						(v_1'-v_1)\epsilon
						\Xi_{vu}(\theta_1,\theta_2)
						+
						(v_1'-v_1)(v_2'-v_2)
						\Xi_{vv}(\theta_1,\theta_2)
					\}.
	\end{multline*}
	Here 
	\begin{align}
		\label{eq849028091843201941}
		\Xi_{xy}(\theta_1,\theta_2)
		&=
			\expect
				[
					\partial_x \tilde{A}_j(u_1(\theta_1),v_1(\theta_1))
					\partial_y \tilde{A}_j(u_2(\theta_2),v_2(\theta_2))
				],		
	\end{align}
	where $x$ and $y$ denote the variables for differentiation.
	Then for $L>0$, we see
	\begin{align}
		\label{eq83821083201}
		|\Xi_{uu}(\theta_1,\theta_2)|
		&\leq
			c
			2^{2j(1-\lambda_\ast-L)}
			(u_2-u_1)^{-2L}
			\{1+a(\theta_2-\theta_1)\}^{-2L},\\
		\label{eq83821083202}
		|\Xi_{uv}(\theta_1,\theta_2)|
		&\leq
			c
			(j+1)
			2^{j\{1-2\lambda_\ast-2L\}}
			(u_2-u_1)^{-2L}
			\{1+a(\theta_2-\theta_1)\}^{-2L},\\
		\label{eq83821083203}
		|\Xi_{vu}(\theta_1,\theta_2)|
		&\leq
			c
			(j+1)
			2^{j\{1-2\lambda_\ast-2L\}}
			(u_2-u_1)^{-2L}
			\{1+a(\theta_2-\theta_1)\}^{-2L},\\
		\label{eq83821083204}
		|\Xi_{vv}(\theta_1,\theta_2)|
		&\leq
			c
			(j+1)^{2}
			2^{-2j(\lambda_\ast+L)}
			(u_2-u_1)^{-2L}
			\{1+a(\theta_2-\theta_1)\}^{-2L},
	\end{align}
	where $a=\epsilon/(u_2-u_1)$.
	Here choose $L=1-\lambda>\frac{1}{2}$.
	Then, since $1<2L<2$, we have
	\begin{align*}
		\sup_{0<a<1}
			\int_{[0,1]^2}
				d\theta_1
				d\theta_2\,
				\{1+a(\theta_2-\theta_1)\}^{-2L}
		<
			\infty.
	\end{align*}
	Actually, for $0<a<\frac{1}{2}$, we have 
	$1+a(\theta_2-\theta_1)>1-\frac{1}{2}=\frac{1}{2}$,
	which implies
	\begin{align*}
		\int_{[0,1]^2}
			d\theta_1
			d\theta_2\,
			\{1+a(\theta_2-\theta_1)\}^{-2L}
		\leq
			2^{2L}.
	\end{align*}
	For $\frac{1}{2}\leq a<1$, direct computation yields
	\begin{align*}
		\int_{[0,1]^2}
			d\theta_1
			d\theta_2\,
			\{1+a(\theta_2-\theta_1)\}^{-2L}
		&=
			\frac{-1}{a^2(1-2L)(2-2L)}
			\left\{
				2
				-(1+a)^{2-2L}
				-(1-a)^{2-2L}
			\right\}\\
		&\leq
			\frac{1}{a^2|1-2L||2-2L|}
			\left\{
				2
				+(1+a)^{2-2L}
				+(1-a)^{2-2L}
			\right\}\\
		&\leq
			\frac{3+2^{2-2L}}{2^{-2}|1-2L||2-2L|}.
	\end{align*}
	Furthermore, putting $\kappa=2(\lambda_\ast-\lambda)$, we have
	\begin{align*}
		2^{2j(1-\lambda_\ast-L)}
		&=
			2^{-j\kappa},\\
		(j+1)
		2^{j\{1-2\lambda_\ast-2L\}}
		&=
			(j+1)
			2^{-j}
			2^{-\kappa j}\\
		(j+1)^{2}
		2^{-2j(\lambda_\ast+L)}
		&=
			(j+1)^2
			2^{-2j}
			2^{-\kappa j}.
	\end{align*}
	Hence, there exists a positive constant $c$ such that
	\begin{align*}
		\int_{[0,1]^2}
				d\theta_1
				d\theta_2\,
				|\Xi_{xy}(\theta_1,\theta_2)|
		&\leq
			c
			2^{-\kappa j}
			(u_2-u_1)^{2\lambda-2},
	\end{align*}
	which implies the assertion. 
	It remains to establish \eqref{eq83821083201}, \eqref{eq83821083202},
	\eqref{eq83821083203}, and \eqref{eq83821083204}.

	Proof of \eqref{eq83821083201}.
	By calculating \eqref{eq849028091843201941} with $x=y=u$,
	we have
	\begin{align*}
		\Xi_{uu}(\theta_1,\theta_2)
		&=
			2^{j\{1-v_1(\theta_1)\}}
			2^{j\{1-v_2(\theta_2)\}}
			\expect
				[
					\partial_y B_j( 2^j u_1(\theta_1), v_1(\theta_1) )
					\partial_y B_j( 2^j u_2(\theta_2), v_2(\theta_2) )
				].
	\end{align*}
	Here, $\partial_y$ denotes the differentiation of $B_j$ in the first variable.
	From \pref{prop32132183190}, we arrive at \eqref{eq83821083201} as follows:
	\begin{align*}
		|\Xi_{uu}(\theta_1,\theta_2)|
		&\leq
			2^{j(2-v_1(\theta_1)-v_2(\theta_2))}
			c|2^j u_1(\theta_1)-2^j u_2(\theta_2)|^{-2L}\\
		&\leq
			c
			2^{2j(1-\lambda_\ast)}
			2^{-2jL}
			| u_2-u_1 + \epsilon (\theta_2-\theta_1)|^{-2L}.
	\end{align*}

	Proof of \eqref{eq83821083202}.
	By calculating \eqref{eq849028091843201941} with $x=u$ and $y=v$,
	we have
	\begin{align*}
		\Xi_{uv}(\theta_1,\theta_2)
		&=
			\expect
				[
					2^{j(1-v_1(\theta_1))}
					\partial_y B_j( 2^j u_1(\theta_1), v_1(\theta_1) )\\
		&\phantom{=}\quad
					\times
					\{
						(-\log 2)j 2^{-jv_2(\theta_2)}
						B_j( 2^j u_2(\theta_2), v_2(\theta_2) )
						+
						2^{-jv_2(\theta_2)}
						\partial_v B_j( 2^j u_2(\theta_2), v_2(\theta_2) )
					\}
				]\\
		&=
			(-\log 2)
			j
			2^{j\{1-v_1(\theta_1)-v_2(\theta_2)\}}
			\expect
				[
					\partial_y B_j( 2^j u_1(\theta_1), v_1(\theta_1) )
					B_j( 2^j u_2(\theta_2), v_2(\theta_2) )
				]\\
		&\phantom{=}\qquad
			+
			2^{j\{1-v_1(\theta_1)-v_2(\theta_2)\}}
			\expect
				[
					\partial_y B_j( 2^j u_1(\theta_1), v_1(\theta_1) )
					\partial_v B_j( 2^j u_2(\theta_2), v_2(\theta_2) )
				].
	\end{align*}
	From \pref{prop32132183190}, we arrive at \eqref{eq83821083202} as follows:
	\begin{align*}
		|\Xi_{uv}(\theta_1,\theta_2)|
		&\leq
			c
			(j+1)
			2^{j\{1-v_1(\theta_1)-v_2(\theta_2)\}}
			|2^j u_1(\theta_1)- 2^j u_2(\theta_2)|^{-2L}\\
		&\leq
			c
			(j+1)
			2^{j\{1-2\lambda_\ast-2L\}}
			| u_2-u_1 + \epsilon (\theta_2-\theta_1)|^{-2L}.
	\end{align*}

	Proof of \eqref{eq83821083203}.
	By calculating \eqref{eq849028091843201941} with $x=v$ and $y=u$,
	we have
	\begin{align*}
		\Xi_{vu}(\theta_1,\theta_2)
		&=
			\expect
				[
					\{
						(-\log 2)j 2^{-jv_1(\theta_1)}
						B_j( 2^j u_1(\theta_1), v_1(\theta_1) )
						+
						2^{-jv_1(\theta_1)}
						\partial_v B_j( 2^j u_1(\theta_1), v_1(\theta_1) )
					\}\\
		&\phantom{=}\qquad\qquad\qquad\qquad
					\times
					2^{j(1-v_2(\theta_2))}
					\partial_y B_j( 2^j u_2(\theta_2), v_2 (\theta_2) )
				]\\
		&=
			(-\log 2)
			j
			2^{j\{1-v_1(\theta_1)-v_2(\theta_2)\}}
			\expect
				[
					B_j( 2^j u_1(\theta_1), v_1(\theta_1) )
					\partial_y B_j( 2^j u_2(\theta_2), v_2 (\theta_2) )
				]\\
		&\phantom{=}\qquad
			+
			2^{j\{1-v_1(\theta_1)-v_2(\theta_2)\}}
			\expect
				[
					\partial_v B_j( 2^j u_1(\theta_1), v_1(\theta_1) )
					\partial_y B_j( 2^j u_2(\theta_2), v_2 (\theta_2) )
				].
	\end{align*}
	From \pref{prop32132183190}, we arrive at \eqref{eq83821083203} as follows:
	\begin{align*}
		|\Xi_{vu}(\theta_1,\theta_2)|
		&\leq
			c
			(j+1)
			2^{j\{1-v_1(\theta_1)-v_2(\theta_2)\}}
			|2^j u_1(\theta_1)- 2^j u_2(\theta_2)|^{-2L}\\
		&\leq
			c
			(j+1)
			2^{j\{1-2\lambda_\ast-2L\}}
			| u_2-u_1 + \epsilon (\theta_2-\theta_1)|^{-2L}.
	\end{align*}

	Proof of \eqref{eq83821083204}.
	By calculating \eqref{eq849028091843201941} with $x=y=v$,
	we have
	\begin{align*}
		\Xi_{vv}(\theta_1,\theta_2)
		&=
			\expect
				[
					\{
						(-\log 2)j 2^{-jv_1(\theta_1)}
						B_j( 2^j u_1(\theta_1), v_1(\theta_1) )
						+
						2^{-jv_1(\theta_1)}
						\partial_v B_j( 2^j u_1(\theta_1), v_1(\theta_1) )
					\}\\
		&\phantom{=}\quad
					\times
					\{
						(-\log 2)j 2^{-jv_2(\theta_2)}
						B_j( 2^j u_2(\theta_2), v_2(\theta_2) )
						+
						2^{-jv_2(\theta_2)}
						\partial_v B_j( 2^j u_2(\theta_2), v_2(\theta_2) )
					\}
				]\\
		&=
			2^{-j\{v_1(\theta_1)+v_2(\theta_2)\}}
			\expect
				[
					\{
						(-\log 2)j 
						B_j( 2^j u_1(\theta_1), v_1(\theta_1) )
						+
						\partial_v B_j( 2^j u_1(\theta_1), v_1(\theta_1) )
					\}\\
		&\phantom{=}\quad
					\times
					\{
						(-\log 2)j 
						B_j( 2^j u_2(\theta_2), v_2(\theta_2) )
						+
						\partial_v B_j( 2^j u_2(\theta_2), v_2(\theta_2) )
					\}
				].
	\end{align*}
	From \pref{prop32132183190}, we arrive at \eqref{eq83821083204} as follows:
	\begin{align*}
		|\Xi_{vv}(\theta_1,\theta_2)|
		&\leq
			c
			2^{-j\{v_1(\theta_1)+v_2(\theta_2)\}}
			(j+1)^2
			|2^j u_1(\theta_1)- 2^j u_2(\theta_2)|^{-2L}\\
		&\leq
			c
			2^{-j\{v_1(\theta_1)+v_2(\theta_2)\}}
			2^{-2 j L}
			(j+1)^2
			| u_2-u_1 + \epsilon (\theta_2-\theta_1)|^{-2L}\\
		&\leq
			c
			(j+1)^2
			2^{-2j(\lambda_\ast+L)}
			| u_2-u_1 + \epsilon (\theta_2-\theta_1)|^{-2L}.
	\end{align*}
	Hence we see \eqref{eq83821083204}.
\end{proof}
\subsection{Proof of Lemma~\ref{lemma8453908409128}}
We are in a position to show \lref{lemma8453908409128}.
\begin{proof}[Proof of \lref{lemma8453908409128}]
	We will show the following:
	\begin{gather}
		\label{eq493820849}
		\expect[|B^\hf_+(t_2)-B^\hf_+(t_1)|^2]
		\leq
			c(t_2-t_1)^{2\lambda}
		\qquad
		\text{for all $t_1,t_2\in I$ with $t_1<t_2$,}\\
		\label{eq490294201}
		|
			\expect
				[
					\{B^\hf_+(t_1+\epsilon)-B^\hf_+(t_1)\}
					\{B^\hf_+(t_2+\epsilon)-B^\hf_+(t_2)\}
				]
		|
		\leq
			c
			(t_2-t_1)^{2\lambda-2}
			\epsilon^2\\
		\notag
		\qquad\qquad\qquad\qquad
		\text{
			for all $t_1,t_2\in I$ and $\epsilon>0$
			with $\epsilon<t_2-t_1$ and $t_1+\epsilon,t_2+\epsilon\in I$.
		}
	\end{gather}
	Here $c$ is a positive constant independent of $t_1$, $t_2$, and $\epsilon$.
	Throughout this proof, we take $K=h(I)$.
	Since $h$ is continuous, $K$ is a compact interval contained in $(0,1)$
	and $\min K=\lambda_*$.

	We show \eqref{eq493820849}.
	Recall $B^\hf_+(t_i)=X^\hf_+(t_i,h(t_i))$ ($i=1,2$).
	We divide $B^\hf_+(t_2)-B^\hf_+(t_1)$ into
	two terms $X^\hf_+(t_2,h(t_2))-X^\hf_+(t_1,h(t_2))$ and $X^\hf_+(t_1,h(t_2))-X^\hf_+(t_1,h(t_1))$
	and give estimates of them. 
	More precisely, for all $t_1,t_2,t\in I$ with $t_1<t_2$, we show
	\begin{align}
		\label{eq493029243242042}
		\expect[\{X^\hf_+(t_2,h(t))-X^\hf_+(t_1,h(t))\}^2]
		&\leq
			c|t_2-t_1|^{2\lambda},\\
		\label{eq121310093213213}
		\expect[\{X^\hf_+(t,h(t_2))-X^\hf_+(t,h(t_1))\}^2]
		&\leq
			c|t_2-t_1|^2.
	\end{align}

	We show \eqref{eq493029243242042}.
	From \eqref{eq493209201} and \pref{prop9842309340},
	we have
	\begin{align*}
		\expect[|X^\hf_+(t_2,h(t))-X^\hf_+(t_1,h(t))|^2]
		&=
			\sum_{j\geq 1}
				\expect[\{\tilde{A}_j(t_2,h(t))-\tilde{A}_j(t_1,h(t))\}^2]\\
		&\leq
			c
			\sum_{j\geq 1}
				2^{-2j(h(t)-\lambda)}
				|t_2-t_1|^{2\lambda}.
	\end{align*}
	Noting
	$2^{-2j(h(t)-\lambda)}\leq 2^{-2j(\lambda_\ast-\lambda)}$ for $j\geq 1$, 
	we arrive at \eqref{eq493029243242042}.

	We show \eqref{eq121310093213213}.
	From \eqref{eq493209201} and \pref{prop438214234143}, we have
	\begin{align*}
		\expect[|X^\hf_+(t,h(t_2))-X^\hf_+(t,h(t_1))|^2]
		&=
			\sum_{j\geq 1}
				\expect[\{\tilde{A}_j(t,h(t_2))-\tilde{A}_j(t,h(t_1))\}^2]\\
		&\leq
			c
			\sum_{j\geq 1}
				(j+1)^2
				2^{-2j\lambda_\ast}
				|h(t_2)-h(t_1)|^2.
	\end{align*}
	Since $h$ is $C^1$, we have the conclusion.

	We show \eqref{eq490294201}.
	From \pref{prop49320431941}, we have
	\begin{multline*}
		|
			\expect
				[
					\{B^\hf_+(t_1+\epsilon)-B^\hf_+(t_1)\}
					\{B^\hf_+(t_2+\epsilon)-B^\hf_+(t_2)\}
				]
		|\\
		\begin{aligned}
			&=
				|
					\expect
						[
							\{X^\hf_+(t_1+\epsilon,h(t_1+\epsilon))-X^\hf_+(t_1,h(t_1))\}
							\{X^\hf_+(t_2+\epsilon,h(t_2+\epsilon))-X^\hf_+(t_2,h(t_2))\}
						]
				|\\
			&=
				\left|
					\sum_{j\geq 1}
						\expect
							[
								\{\tilde{A}_j(t_1+\epsilon,h(t_1+\epsilon))-\tilde{A}_j(t_1,h(t_1))\}
								\{\tilde{A}_j(t_2+\epsilon,h(t_2+\epsilon))-\tilde{A}_j(t_2,h(t_2))\}
							]
				\right|\\
			&\leq
				\sum_{j\geq 1}
					c
					2^{-\kappa j}
					(t_2-t_1)^{2\lambda-2}
					\big\{
						|h(t_1+\epsilon)-h(t_1)|+\epsilon
					\big\}
					\big\{
						|h(t_2+\epsilon)-h(t_2)|+\epsilon
					\big\}\\
			&\leq
				c
				(t_2-t_1)^{2\lambda-2}
				\epsilon^2.
		\end{aligned}
	\end{multline*}
	This completes the proof.
\end{proof}

\appendix
\section{Technical lemmas}\label{SecAppendix}
In this appendix, we show lemmas used in Section~\ref{SecAnalysisOfGenerator}.
We denote by $\zeta$ the zeta function.
\begin{lemma}\label{lemma849208094238902}
	Let $\lambda>1$. 
	For all $x\in\RealNum$, we have
	\begin{align}
		\label{eq43i2oi4143}
		\sum_{k\in\Integers}
			(2+|x-k|)^{-\lambda}
		\leq
			2\zeta(\lambda).
	\end{align}
\end{lemma}
\begin{proof}
	We denote by $f(x)$ the sum on the left-hand side of \eqref{eq43i2oi4143}.
	Since $f$ is $1$-periodic, it suffices to consider the case $0\leq x<1$.
	Under this assumption, we consider the following cases: $k\leq 0$ and $k\geq 1$.
	Let $k\leq 0$. Then, we have
	$|x-k|=x-k\geq -k$, which implies
	\begin{align*}
		\sum_{k\leq 0}
			(2+|x-k|)^{-\lambda}
		\leq
			\sum_{k\leq 0}
			(2+(-k))^{-\lambda}
		=
			\sum_{k\geq 2}
				k^{-\lambda}
		\leq
			\zeta(\lambda).
	\end{align*}
	Let $k\geq 1$.
	Then, we have
	$|x-k|=k-x\geq k-1$, which implies
	\begin{align*}
		\sum_{k\geq 1}
			(2+|x-k|)^{-\lambda}
		\leq
			\sum_{k\geq 1}
				(k+1)^{-\lambda}
		=
			\sum_{k\geq 2}
				k^{-\lambda}
		\leq
			\zeta(\lambda).
	\end{align*}
	From the above, we have $f(x)\leq 2\zeta(\lambda)$.
	This completes the proof.
\end{proof}

\begin{lemma}\label{lem483901134}
	Let $\kappa,\lambda>0$. Assume that $\kappa+\lambda>1$ holds. 
	For all distinct $x,y\in\RealNum$, we have
	\begin{align}
		\label{eq4832908019}
		\sum_{k\in\Integers}
			(2+|x-k|)^{-\kappa-\lambda}
			(2+|y-k|)^{-\kappa-\lambda}
		\leq
			2(1+2^\lambda)\zeta(\kappa+\lambda)
			|y-x|^{-\lambda}.
	\end{align}
\end{lemma}
\begin{proof}
	We denote by $f(x,y)$ the sum on the left-hand side of \eqref{eq4832908019}.
	Since $f(x,y)=f(y,x)$ and $f(x,y)=f(x+p,y+p)$ for all $p\in\Integers$,
	it suffices to consider the case $0\leq x<1$ and $x<y$.

	Let $k\leq 0$. Then we have
	\begin{align*}
		|x-k|&=x-k\geq -k, & 
		|y-k|&=y-k=(y-x)+(x-k)\geq y-x,
	\end{align*}
	which implies
	\begin{align*}
		\sum_{k\leq 0}
			(2+|x-k|)^{-\kappa-\lambda}
			(2+|y-k|)^{-\kappa-\lambda}
		&\leq
			\sum_{k\leq 0}
				(2+(-k))^{-\kappa-\lambda}
				(2+y-x)^{-\kappa-\lambda}\\
		&\leq
			\zeta(\kappa+\lambda)
			(y-x)^{-\lambda}.
	\end{align*}

	Let $1\leq k\leq \left\lfloor \frac{y-x}{2} \right\rfloor$. Then we have
	\begin{align*}
		|x-k|&=(k-1)+(1-x)\geq k-1, \\
		|y-k|
		&=
			y-k
		=
			\frac{y-x}{2}
			+
			\left(
				\frac{y-x}{2}
				-
				\left\lfloor \frac{y-x}{2} \right\rfloor
			\right)
			+
			\left(
				\left\lfloor \frac{y-x}{2} \right\rfloor
				-
				k
			\right)
			+
			x
		\geq
			\frac{y-x}{2},
	\end{align*}
	which implies
	\begin{align*}
		\sum_{k=1}^{\left\lfloor \frac{y-x}{2} \right\rfloor}
			(2+|x-k|)^{-\kappa-\lambda}
			(2+|y-k|)^{-\kappa-\lambda}
		&\leq
			\sum_{k=1}^{\left\lfloor \frac{y-x}{2} \right\rfloor}
				(k+1)^{-\kappa-\lambda}
				\left(
					2+\frac{y-x}{2}
				\right)^{-\kappa-\lambda}\\
		&\leq
			\zeta(\kappa+\lambda)
			2^\lambda
			(y-x)^{-\lambda}.
	\end{align*}

	Let $\left\lfloor \frac{y-x}{2} \right\rfloor+1\leq k\leq \floor{y}$. Then we have
	\begin{align*}
		|x-k|
		&=
			k-x\\
		&=
			\left\{
				k
				-
				\left(
					\left\lfloor \frac{y-x}{2} \right\rfloor
					+1
				\right)
			\right\}
			+
			\left\{
				\left(
					\left\lfloor \frac{y-x}{2} \right\rfloor
					+1
				\right)
				-
				\frac{y-x}{2}
			\right\}
			+
			\frac{y-x}{2}
			-
			x \\
		&\geq
			\frac{y-x}{2}
			-
			1, \\
		|y-k|
		&=
			y-k
		=
			(y-\floor{y})
			+
			(\floor{y}-k)
		\geq
			\floor{y}-k,
	\end{align*}
	which implies
	\begin{multline*}
		\sum_{k=\left\lfloor \frac{y-x}{2} \right\rfloor+1}^{\floor{y}}
			(2+|x-k|)^{-\kappa-\lambda}
			(2+|y-k|)^{-\kappa-\lambda}\\
		\leq
			\sum_{k=\left\lfloor \frac{y-x}{2} \right\rfloor+1}^{\floor{y}}
				\left(
					\frac{y-x}{2}
					+
					1
				\right)^{-\kappa-\lambda}
				\left(
					2+\floor{y}-k
				\right)^{-\kappa-\lambda}
		\leq
			\zeta(\kappa+\lambda)
			2^\lambda
			(y-x)^{-\lambda}.
	\end{multline*}

	Let $\floor{y}+1\leq k$.
	Then we have
	\begin{align*}
		|x-k|&=k-x\geq y-x, \\
		|y-k|
		&=
			k-y
		=
			k-(\floor{y}+1)+(1+\floor{y}-y)
		\geq
			k-\floor{y}-1,
	\end{align*}
	which implies
	\begin{align*}
		\sum_{\floor{y}+1\leq k}
			(2+|x-k|)^{-\kappa-\lambda}
			(2+|y-k|)^{-\kappa-\lambda}
		&\leq
			\sum_{\floor{y}+1\leq k}
				(2+y-x)^{-\kappa-\lambda}
				(k-\floor{y}+1)^{-\kappa-\lambda}\\
		&\leq
			\zeta(\kappa+\lambda)
			(y-x)^{-\lambda}.
	\end{align*}
	From the above, we arrive at the assertion.
\end{proof}


\begin{thebibliography}{10}

\bibitem{Ayache2019Book}
Antoine Ayache.
\newblock {\em Multifractional stochastic fields}.
\newblock World Scientific Publishing Co. Pte. Ltd., Hackensack, NJ, 2019.

\bibitem{BenassiJaffardRoux1997}
Albert Benassi, St\'ephane Jaffard, and Daniel Roux.
\newblock Elliptic {G}aussian random processes.
\newblock {\em Rev. Mat. Iberoamericana}, 13(1):19--90, 1997.

\bibitem{BiaginiHuOksendalZhang2008Book}
Francesca Biagini, Yaozhong Hu, Bernt {\O}ksendal, and Tusheng Zhang.
\newblock {\em Stochastic calculus for fractional {B}rownian motion and applications}.
\newblock Probability and its Applications (New York). Springer-Verlag London, Ltd., London, 2008.

\bibitem{BianchiPantanellaPianese2012}
Sergio Bianchi, Alexandre Pantanella, and Augusto Pianese.
\newblock Modeling and simulation of currency exchange rate using multifractional process with random exponent.
\newblock {\em International Journal of Modeling and Optimization}, 2(3):309--314, 2012.

\bibitem{BianchiPantanellaPianese2013}
Sergio Bianchi, Alexandre Pantanella, and Augusto Pianese.
\newblock Modeling stock prices by multifractional {B}rownian motion: an improved estimation of the pointwise regularity.
\newblock {\em Quant. Finance}, 13(8):1317--1330, 2013.

\bibitem{BianchiPianese2014}
Sergio Bianchi and Augusto Pianese.
\newblock Multifractional processes in finance.
\newblock {\em Risk and Decision Analysis}, 5(1):1--22, 2014.

\bibitem{CoutinQian2002}
Laure Coutin and Zhongmin Qian.
\newblock Stochastic analysis, rough path analysis and fractional {B}rownian motions.
\newblock {\em Probab. Theory Related Fields}, 122(1):108--140, 2002.

\bibitem{FrizGessGulisashviliRiedel2016}
Peter~K. Friz, Benjamin Gess, Archil Gulisashvili, and Sebastian Riedel.
\newblock The {J}ain-{M}onrad criterion for rough paths and applications to random {F}ourier series and non-{M}arkovian {H}\"ormander theory.
\newblock {\em Ann. Probab.}, 44(1):684--738, 2016.

\bibitem{FrizHairer2020Book}
Peter~K. Friz and Martin Hairer.
\newblock {\em A course on rough paths}.
\newblock Universitext. Springer, Cham, Second edition, 2020.

\bibitem{FrizVictoir2010b}
Peter~K. Friz and Nicolas~B. Victoir.
\newblock Differential equations driven by {G}aussian signals.
\newblock {\em Ann. Inst. Henri Poincar{\'e} Probab. Stat.}, 46(2):369--413, 2010.

\bibitem{FrizVictoir2010Book}
Peter~K. Friz and Nicolas~B. Victoir.
\newblock {\em Multidimensional stochastic processes as rough paths}, volume 120 of {\em Cambridge Studies in Advanced Mathematics}.
\newblock Cambridge University Press, Cambridge, 2010.

\bibitem{Harang2020}
Fabian~A. Harang.
\newblock Differential equations driven by variable order {H}{\"o}lder noise and the regularizing effect of delay.
\newblock {\em Stochastics}, 92(1):90--123, 2020.

\bibitem{InahamaNaganuma2021}
Yuzuru Inahama and Nobuaki Naganuma.
\newblock Asymptotic expansion of the density for hypoelliptic rough differential equation.
\newblock {\em Nagoya Math. J.}, 243:11--41, 2021.

\bibitem{Lyons1998}
Terry~J. Lyons.
\newblock {Differential equations driven by rough signals}.
\newblock {\em Rev. Mat. Iberoamericana}, 14(2):215--310, 1998.

\bibitem{LyonsQian2002Book}
Terry~J. Lyons and Zhongmin Qian.
\newblock {\em System control and rough paths}.
\newblock Oxford Mathematical Monographs. Oxford University Press, Oxford, 2002.

\bibitem{Mishura2008Book}
Yuliya~S. Mishura.
\newblock {\em Stochastic calculus for fractional {B}rownian motion and related processes}, volume 1929 of {\em Lecture Notes in Mathematics}.
\newblock Springer-Verlag, Berlin, 2008.

\bibitem{Nourdin2012Book}
Ivan Nourdin.
\newblock {\em Selected aspects of fractional {B}rownian motion}, volume~4 of {\em Bocconi \& Springer Series}.
\newblock Springer, Milan; Bocconi University Press, Milan, 2012.

\bibitem{PeltierLevyVehel1995}
Romain-François Peltier and Jacques L{\'e}vy~V{\'e}hel.
\newblock Multifractional {B}rownian motion: Definition and preliminary results.
\newblock 1995.

\bibitem{PipirasTaqqu2017Book}
Vladas Pipiras and Murad~S. Taqqu.
\newblock {\em Long-range dependence and self-similarity}, volume [45] of {\em Cambridge Series in Statistical and Probabilistic Mathematics}.
\newblock Cambridge University Press, Cambridge, 2017.

\bibitem{YoungLC1936}
Laurence~C. Young.
\newblock An inequality of the {H}\"{o}lder type, connected with {S}tieltjes integration.
\newblock {\em Acta Math.}, 67(1):251--282, 1936.

\end{thebibliography}
\end{document}